\documentclass[11pt,oneside]{amsart}
\usepackage{amsmath}
\usepackage{amssymb}
\usepackage{amscd}
\usepackage{amsthm}
\usepackage{amsxtra}
\usepackage{latexsym}
\usepackage{enumitem}
\usepackage{graphicx}
\usepackage{comment}
\usepackage{tikz,tikz-3dplot}

\usepackage[all]{xy}
\numberwithin{equation}{section}
\theoremstyle{plain}
\newtheorem{theorem}[equation]{Theorem}

\newtheorem{proposition}[equation]{Proposition}

\newtheorem{lemma}[equation]{Lemma}

\theoremstyle{remark}

\theoremstyle{definition}
\newtheorem{definition}[equation]{Definition}

\newtheorem{example}[equation]{Example}

\DeclareMathOperator{\Spec}{Spec}

\title{The dual complex of a semi-log canonical surface}
\author{Morgan V Brown}
\address{Department of Mathematics, University of Miami, Coral Gables, FL 33146 USA}
\email{mvbrown@math.miami.edu}

\begin{document}

\begin{abstract}
Semi-log canonical varieties are a higher-dimensional analogue of stable curves. They are the varieties appearing as the boundary $\Delta$ of a log canonical pair $(X,\Delta)$, and also appear as limits of canonically polarized varieties in moduli theory. For certain three-fold pairs $(X,\Delta)$ we show how to compute the PL homeomorphism type of the dual complex of a dlt minimal model directly from the normalization data of $\Delta$. 

\end{abstract}
\maketitle
\section{Introduction}
Let $X$ be a smooth variety over $\mathbb{C}$, and let $\Delta$ be a simple normal crossings (snc) divisor. We associate to the pair $(X,\Delta)$ the cell complex $\mathcal{D}(\Delta)$, which is the divisor's dual intersection complex. These complexes appear ubiquitously in algebraic geometry with connections to Hodge theory, moduli theory, birational geometry, non-archimedean geometry, and the mirror symmetry program.

Pairs and dual complexes arise naturally in the context of degenerations. Consider a one parameter family of projective varieties $X$ over a punctured curve $S \subset \bar{S}$. Even if $X$ is smooth over $S$, there may be no way to complete the family to a smooth family over $\bar{S}$. But the existence of log resolutions guarantees a compactification $\bar{X}$ over $\bar{S}$ such that the pair $(\bar{X},X_0)$ has simple normal crossings. Our motivation is the natural question of how the dual complex $\mathcal{D}(X_0)$ encodes the geometry of the generic fiber. For example, if $X$ is a family of smooth genus $g$ curves, the dual complex of the special fiber is either a point or a graph of genus at most $g$. For K3 surfaces, there are three possible types, distinguished by the monodromy action on the cohomology groups of the variety. If $\bar{X}$ is a degeneration of K3 surfaces with trivial relative canonical bundle and reduced special fiber, the dual complex $\mathcal{D}(X_0)$ is either a point, an interval, or a $2$-sphere \cite{Kulikov}.

As different compactifications of the same family are related by a birational map, recent papers have employed techniques from minimal model theory to study the dual complex. In this setting it is natural to extend the definitions to mildly singular pairs such as dlt pairs as in \cite{deFernexKollarXu2012} as long as we only consider divisors of coefficient $1$. For a dlt pair $(X,\Delta)$, de Fernex, Koll{\'a}r, and Xu \cite{deFernexKollarXu2012} show that certain operations of the minimal model program induce simple homotopy equivalences on the complex $\mathcal{D}(\Delta^{=1})$. As a result they are able to show that if $X \to S$ is a one-parameter degeneration of rationally connected varieties with special fiber $X_0$, if the pair $(X,X_0)$ is dlt, then the dual complex $\mathcal{D}(X_0)$ is contractible. They also show that that the PL homeomorphism type of $\mathcal{D}(\Delta^{=1})$ is preserved by log crepant maps. Koll{\'a}r and Xu \cite{KollarXu} use these techniques to study the special fibers of degenerations of Calabi-Yau varieties. They show that if the dual complex has real dimension $n$ equal to the complex dimension of a fiber, then for small $n$ ($\leq 3$ in the dlt case, $\leq 4$ in the snc case), the dual complex is a sphere.

To understand degenerations of varieties of general type, we include more singular pairs called \emph{log canonical} pairs. Now consider a $1$-parameter degeneration $X \to S$, where the generic fibers of $X$ are smooth surfaces with ample canonical bundle. Then because the KSBA moduli space is proper \cite{Alexeev, KollarShepherd-Barron} possibly after a finite base change, we may replace $X$ by $\bar{X}$, such that $X_0$ is reduced, $(\bar{X}, X_0)$ is a log canonical pair, and the relative canonical bundle $K_{\bar{X}/\bar{S}}$ is ample.

The special fiber $X_0$ itself does not always have log canonical singularities because it is not necessarily normal. As it may have multiple components, we will consider a complex variety to be a finite type reduced scheme over $\mathbb{C}$ which may have multiple components. The singular variety $X_0$ is a \emph{semi-log canonical} variety, which guarantees that each irreducible component has at worst nodal singularities in codimension one, though there may be more complicated self intersections in higher codimension.

Any log canonical pair $(X,\Delta)$ admits a \emph{minimal dlt model}, $(X', \Delta')$. As any two choices of dlt minimal model give dual complexes which are PL homeomorphic, we associate to the pair $(X,\Delta)$ the PL homeomorphism type $\mathcal{DMR}(X,\Delta)$ of the complex $\mathcal{D}(\Delta'^{=1})$. The goal of the present work is to compute this complex for a threefold pair directly from the geometry of the divisor $\Delta$. This has been accomplished in arbitrary dimension when the boundary divisor has normal crossings in the context of tropicalization of the moduli space of curves \cite{ACP2015,2018arXiv180510186C}

In our setting $\Delta$ may have more complicated self intersections in codimension $2$. Suppose every element of $\Delta$ has coefficient $1$, $\Delta$ is Cartier, and the normalization of $\Delta$ with its conductor forms a dlt pair. Under these hypotheses our main result shows that $\mathcal{D}(\Delta')$ is PL homeomorphic to a complex $\mathcal{C}(\Delta)$ depending only on the semi-log canonical variety $\Delta$:

\begin{theorem}\label{dmr}
Let $(X,\Delta)$ be a threefold log canonical pair such that every coefficient of $\Delta$ is $1$, $\Delta$ is Cartier, and every lc center of the pair is contained in $\Delta$. Choose a dlt minimal model $(X',\Delta')$ over $(X,\Delta)$. If $(D,B)$, the normalization of $\Delta$ with its conductor, is a dlt surface pair, then the complexes $\mathcal{D}(\Delta')$ and $\mathcal{C}(\Delta)$ are PL homeomorphic.
\end{theorem}

In section \ref{slc} we review singularities of the minimal model program including semi-log canonical singularities, following Koll{\'a}r's book \cite{KollarBook}. Section \ref{dualcx} introduces necessary background from combinatorial topology. The main result with its proof is presented in section \ref{mainsection}. In the final section we apply Theorem 1.1 to a particular semi-log canonical surface investigated by Franciosi, Pardini, and Rollenske \cite{FPR2015a,FPR2015b,FPR2017}, which appears as a stable limit of I-surfaces, which are surfaces of general type with $K^2=1, p_g=2, q=0$.

This project grew out of a question posed to the author by Phillip Griffiths. The author is also grateful to Mark Green and Ken Baker for helpful discussions. The author's work was supported by Simons Foundation Collaboration Grant 524003.



\section{Semi-log canonical varieties} \label{slc}

We recall a few definitions from the minimal model program.

\begin{definition}\cite{KollarMori}
We say that a pair $(X,\Delta)$ is log canonical if 
\begin{itemize}
\item $X$ is normal.
\item $K_{X}+\Delta$ is $\mathbb{Q}$-Cartier.
\item For any log resolution $f \colon X' \to X$, if we write $K_X'+\Delta'= f^*(K_X+\Delta)+a_i E_i$, where $\Delta'$ is the reduced exceptional divisor plus the strict transform of $\Delta$ with coefficient $1$, each $a_i\geq 0$.
\end{itemize}
\end{definition}

In fact it is sufficient to check the last condition for a single log resolution. The coefficients $a_i$ are called log discrepancies and they provide a measure of the singularities of the pair. In particular we say that a closed subset $Z \subset X$ is a \emph{log canonical center} of $(X,\Delta)$ if it is the image of an $E_i$ with $a_i=0$. A pair $(X,\Delta)$ is called \emph{divisorially log terminal}, or \emph{dlt}, if it is log canonical and is snc in a neighborhood of any log canonical center. 

We say that $(X',\Delta'), f \colon X' \to X$ is a dlt minimal model of the log canonical pair $(X,\Delta)$ if $(X', \Delta')$ is a dlt pair, $\Delta$ is the reduced exceptional divisor of $f$ plus the strict transform of $\Delta$, and every exceptional divisor of $f$ has log discrepancy $0$. By a theorem of Hacon \cite[Theorem 3.1]{KollarKovacs} every log canonical pair has a dlt minimal model. Dlt minimal models are generally not unique.

The condition that $(X,\Delta)$ is log canonical also constrains the singularities of $\Delta$. We cannot guarantee that $\Delta$ is log canonical because $\Delta$ often fails to be normal. This is the motivation for the notion of a semi-log canonical variety. Suppose every component of $\Delta$ has coefficient $1$. Let $f \colon D \to X$ be the normalization map of $\Delta$. There is a canonically defined divisor $M$ on $D$, called the \emph{different}, such that $K_D+M= f^*(K_X+\Delta)$ as $\mathbb{Q}$-divisors \cite[Prop 4.5]{KollarBook}. Moreover if $(X,\Delta)$ is log canonical so is $(D,M)$ \cite[Lemma 4.8]{KollarBook}.

We suppose $\Delta$ is a variety which satisfies $S_2$ and has at most nodal singularities in codimension $1$. Write $D$ for the normalization of $\Delta$, and $C$ for the conductor, which is the divisor on $D$ which is the preimage of the double locus of $\Delta$. The normalization $D \to \Delta$ identifies points of $B$, via an involution $\iota$ on the normalization $B^\nu$ of $B$.

\begin{definition}(see \cite[Definition-Lemma 5.10]{KollarBook})
A variety $\Delta$ as above is \emph{semi-log canonical} if 
\begin{enumerate}
\item The canonical sheaf of $\Delta$ is $\mathbb{Q}$-Cartier
\item The pair $(D,B)$ is log canonical.
\end{enumerate}

\end{definition}

The canonical sheaf of the non-normal variety $\Delta$ is understood as follows. Let $j \colon \Delta_0 \to \Delta$ be the locus where $\Delta_0$ is locally a complete intersection. Then $\omega_0$, the canonical sheaf of $\Delta_0$ is a line bundle, and the complement of $\Delta_0$ has codimension at least $2$ in $\Delta$ because the codimension $1$ singularities are nodal. The canonical sheaf is defined to be $j_*\omega_0$. Surprisingly, if the second condition of Definition \ref{slc} holds for a quasiprojective variety, the first holds if and only if $\iota$ preserves the different of $B^\nu$ \cite[Theorem 5.38]{KollarBook}. 

If $\Delta$ is Cartier in codimension $3$, then the different and the conductor agree, and thus whenever $(X,\Delta)$ is log canonical, $\Delta$ is slc. We will only be concerned with the case where $\Delta$ is a surface and $(D,B)$ is dlt. In particular $\iota$ must preserve the log canonical centers of $(D,B)$.

Semi-log canonical varieties are the higher dimensional analogue of stable curves. The limit of a family of canonically polarized surfaces in the KSBA space is a semi-log canonical surface.

\begin{example}\label{doubleumbrella}
Let $S$ be the affine surface with coordinate ring $\mathbb{C}[s^2 ,st, t^2, s^2t, st^2]$. The normalization of $S$ is the affine plane $\mathbb{A}^2$, with conductor consisting of the two coordinate lines $V(s), V(t)$. The involution on each curve is given by $s \to -s$, $t \to -t$ on each curve. Hence the origin is fixed by the involution, and we see that $S$ is semi-log canonical. We may extract a divisor at the origin to produce $S'$, which is two copies of the Whitney umbrella glued together.

While $S$ is not Gorenstein, the canonical divisor of $S$ is $2$-torsion.

\end{example}

\section{Dual Complexes}\label{dualcx}

The basic building block of the dual complex of a pair is the standard $n$-simplex, 

\[
\Sigma_n = \{(x_0, \ldots x_n) \in \mathbb{R}^{n+1} | x_i \geq 0, x_0 + \ldots x_n =1\}.
\]
The boundary $\partial\Sigma_n$ consists of the $n-1$-simplices $x_i =0$ for $i=0 \ldots n$.

Following \cite{deFernexKollarXu2012} and \cite{Hatcher}, we define a (finite) regular cell complex inductively as follows: A $0$-dimensional cell complex consists of finitely many points called vertices. Given a $k$-dimensional regular cell complex $\mathcal{C}_k$, we may attach finitely many $k+1$ simplices via embedding their boundaries in $\mathcal{C}_k$. At each stage the $j$-cells are the images of the $j$ dimensional simplices. If the intersection of any two simplices is a (possibly empty) face of both then we say that the cell complex is a simplicial complex.

Let $(X,\Delta)$ be a dlt pair. The dual complex $\mathcal{D}(\Delta^{=1})$ is the regular cell complex whose vertices are components of $\Delta^{=1}$, and whose $j$-cells are irreducible components of intersections of $j+1$ many components of $\Delta^{=1}$, which are the log canonical centers of $(X,\Delta)$ of codimension $j+1$. The attachment maps arise from the fact that every log canonical center $Z$ of codimension $j+1$ is contained in $j+1$ many lc centers of codimension $j$.

By \cite[Prop 11]{deFernexKollarXu2012}, if $(X,\Delta)$ and $(X',\Delta')$ are birational dlt pairs which satisfy $f^*(K_X+\Delta)=f'^*(K_{X'}+\Delta')$ for a common resolution, then the complexes $\mathcal{D}(\Delta^{=1})$ and $\mathcal{D}(\Delta'^{=1})$ are PL homeomorphic. As a result we can define a PL space for any pair $(X,\Delta)$ by taking $\mathcal{D}(\Delta'^{=1})$ for a dlt minimal model $(X,\Delta)$. Following the notation in \cite{deFernexKollarXu2012}, we call this $\mathcal{DMR}(X,\Delta)$.

It is often more convenient to work with simplicial complexes than regular cell complexes. We introduce the \emph{barycentric subdivision}, which replaces a complex $\mathcal{C}$ by its flag complex $\mathcal{B}(\mathcal{C})$. The vertices of $\mathcal{B}(\mathcal{C})$ correspond to the simplices of $\mathcal{C}$, and the $j$-simplices of $\mathcal{B}(\mathcal{C})$ are length $j$ chains of cells $C_1 \subsetneq C_2 \ldots \subsetneq C_j$. A complex and its barycentric subdivision are PL homeomorphic, so for a dlt pair $(X,\Delta)$ we could have defined $\mathcal{D}(\Delta^{=1})$ to be the flag complex of the lc centers of $(X,\Delta)$, considered as a partially ordered set.

While it would be nice to define $\mathcal{D}(X,\Delta)$ for log canonical pairs the same way, this gives the wrong topological type in certain cases: Suppose $X$ is a smooth threefold, and $\Delta$ a surface such that $(X,\Delta)$ is log canonical, and the only log canonical centers of $\Delta$ are $\Delta$ itself and a smooth curve. Then the flag complex is homeomorphic to a closed interval. Consider the normalization with the conductor, $(D, B)$. The curve has either one or two components. We will see (Lemma \ref{graphs}) that if $B$ has one component then $\mathcal{D}(X,\Delta)$ is an interval, but if $B$ has two components $\mathcal{D}(\Delta)$ is a circle.

There are some subtleties to working in the PL category. Namely, two simplicial complexes may be homeomorphic as topological spaces but not PL homeomorphic. The standard example is to take a non simply connected $\mathbb{Z}$-homology $3$-sphere $X$, such as the Poincar{\'e} sphere, and let $Y$ be the double suspension over it. As a topological space, $Y$ is homeomorphic to $S^5$, but it is not PL homeomorphic to the boundary of the standard $6$ simplex \cite{Cannon}. We work in dimension $2$, and rely on the fact that any simplicial complex which is homeomorphic to the $2$-disk is also PL homeomorphic to the standard $2$-simplex.

Finally, we introduce the notion of a half-edge graph, which will be useful for comparing various cell complexes of dimension $1$.

\begin{definition}
A half edge graph consists of a topological space $U$, a finite set of distinct points $p_i \in U$ and finitely many line segments 
$L_j \colon [0,1] \to U$ such that 
\begin{enumerate}
\item $U$ is the union of the images of the $L_i$.
\item For each $L_j$, $L_j(0)$ is one of the $p_i$, and for $x>0$, $L_j(x) \neq p_i$ for any $i$.
\item For each $L_j$ there is at most one $k \neq j$ such that $L_j((0,1]) \cap L_k((0,1])$ is nonempty, and if so the intersection consists of the single point $L_j(1)=L_k(1)$.
\end{enumerate}

We say that two half edge graphs $(U, p_i, L_j)$ and $(U', p_i', L_j')$ are isomorphic if there is a homeomorphism $f\colon U \to U'$ which induces a bijection on points $p_i$ and half edges $L_j$, and $f$ restricts to a homeomorphism between identified half edges. The positive part of a half edge is the image of the interval $(0,1]$.

\end{definition}

\section{Main Result}\label{mainsection}
Let $(X,\Delta)$ be a threefold log canonical pair such that every coefficient of $\Delta$ is $1$, $\Delta$ is Cartier, and every lc center of the pair is contained in $\Delta$. Suppose also that for each component $D_i$ of the normalization of $\Delta$, the pair $(D_i,B_i)$ is dlt, where $B_i$ is the conductor. We will construct a simplicial complex $\mathcal{C}(\Delta)$ from the normalization data.

For every component $\Delta_i$ of $\Delta$, let $(D_i, B_i)$ be the normalization with its conductor. Set $\mathcal{G}_i$ to be the first barycentric subdivision of the dual graph, and $\mathcal{C}_i$ the cone over $\mathcal{G}_i$. Set $\mathcal{C}=\coprod \mathcal{C}_i$ and $\mathcal{G} = \coprod \mathcal{G}_i$, considered as a subcomplex of $\mathcal{C}$. We give an equivalence relation $\sim$ on $\mathcal{G}$ as follows. Two vertices of $\mathcal{G}$ are identified if their strata have the same image in $X$. Two edges of $\mathcal{G}$ are identified if and only if the endpoints of that edge are identified. As $\mathcal{G} \subset \mathcal{C}$ the equivalence relation extends to $\mathcal{C}$.

\begin{definition}
$\mathcal{C}(\Delta) = \mathcal{C}/\sim$.
\end{definition}

The vertices of $\mathcal{C}(\Delta)$ correspond naturally to the lc centers of $(X,\Delta)$, and each $k$-simplex of $\mathcal{C}(\Delta)$ can be identified with some length $k+1$ chain of strata in the normalization. We denote by $\mathcal{C}_1(\Delta)$ the subcomplex consisting of the simplices which involve only strata of dimension at least $1$. This complex naturally has the structure of a half edge graph as follows: The vertices correspond to components of $\Delta$. For every $1$-dimensional lc center $\Gamma$ there is either one or two half edges, depending on the normalization data.
\begin{enumerate}
\item If $\Gamma$ is the intersection of two components $\Delta_1, \Delta_2$, it has two preimages $C_{j}, C_{j'}$ in the normalization, which give two half edges connecting the vertices $D_1$ and $D_2$.

\item If $\Gamma$ is the self intersection of $\Delta_1$, and two irreducible components of the conductor map to $\Gamma$, then there are two half edges corresponding to each component which give an edge connecting $D_1$ to itself.

\item If $\Gamma$ is the self intersection of $\Delta_1$, and the normalization restricts to a double cover of an irreducible curve $B$ onto $\Gamma$, we have a half edge extending from $D_1$, terminating at $C_j$.

\end{enumerate}

These are the only three possible cases. In each case we orient each half edge so that $0$ corresponds to the divisor and $1$ corresponds to the curve. To obtain the rest of $\mathcal{C}(\Delta)$ from the half edge graph $\mathcal{C}_1(\Delta)$ we attach simplices involving the preimages of the $0$-dimensional lc centers. We will see that for each $0$-dimensional center, we attach a closed disk to $\mathcal{C}_1(\Delta)$. The gluing data consists of either a path or cycle in $\mathcal{C}_1(\Delta)$.

Choose a $0$-dimensional lc center $Z$. We associate a graph to $Z$. Let $Q_1 \ldots Q_m$ be the preimages of $Z$ in $D$. Each $Q_k$ is contained in a unique component $D_i$, and there are exactly two curve strata $C_{j}, C_{j'}$ containing $Q_k$. Let $\mathcal{G}(Z)$ be the half edge graph with vertices $Q_j$, and a half edge for every inclusion $Q_k \subset C_j$. Two half edges $Q_k \subset C_j$ and $Q_{k'} \subset C_{j'}$ are glued at $1$ if $\iota$ identifies $C_{j}$ and $C_{j'}$ and $\iota|_{B_{j}}(Q_j)=Q_{j'}$.

This graph is connected because the identification of the $Q_k$ is induced by $\iota$. Each vertex has degree at most $2$ so $\mathcal{G}(Z)$ is homeomorphic to either $S^1$ or a closed interval.


Then we have a map 
\[
\sigma_{Z} \colon \mathcal{G}(Z) \to \mathcal{C}_1(X,\Delta)
\]
which sends the half edge $(Q_k,C_j)$ to the half edge $(D_i, C_j)$, where $D_i$ is the normalization of the divisor containing $Q_k$.

\begin{proposition}\label{attachment}The complex $\mathcal{C}(X,\Delta)$ is obtained from $\mathcal{C}_1(X,\Delta)$ by gluing the cone over $\mathcal{G}(Z)$ along the map $\sigma_Z$, for each $0$-dimensional lc center $Z$.
\end{proposition}

\begin{proof}
We obtain $\mathcal{C}(X,\Delta)$ from $\mathcal{C}(X,\Delta)$ by attaching all of the $2$-simplices, which correspond to chains of strata $Q_k \subset C_j \subset D_i$ in the normalization. We view such a simplex as the cone over the half edge $(Q, C)$ in the graph $\mathcal{G}'(Z)$. Two simplices $Q_{k} \subset C_{j} \subset D_{i}$ and $Q_{k'} \subset C_{j'} \subset D_{k'}$ intersect along the $(Q,D)$ edge if and only if $Q_{k}=Q_{k'}$ and $D_i=D_{i'}$, which is the case if $C_j$ and $C_{j'}$ are the two curves in $D_i$ containing $Q_k$. This occurs exactly when the two half edges $Q_k \subset C_j, Q_{k'} \subset C_{j'}$ are joined at $Q_k$ in $\mathcal{G}'(Z)$. 

The two simplices intersect along the $(Q,C)$ edge if and only if the involution $\iota$ identifies $C_{j}$ and $C_{j'}$ and $\iota|_{C_j}(Q_k)=Q_{k'}$. In this case the two half edges $(Q_k, C_j)$ and $(Q_{k'}, C_{j'})$ of $\Gamma'(P)$ are attached by the identification of $C_j$ and $C_{j'}$ via the involution.

We have shown that attaching all simplices corresponding to chains containing a preimage of $P$ means attaching the cone over $\mathcal{G}(Z)$ to $\mathcal{C}_1(\Delta)$. The attaching map identifies each half edge $(Q_k,C_j)$, with the half edge $(D_i,C_j)$ where $D_i$ is the unique divisor containing $Q_k$. This induces the map $\sigma_Z$ and the proposition follows.

\end{proof}

Our main result, Theorem \ref{dmr}, is that under these hypotheses the complex $\mathcal{C}(\Delta)$, which depends only on $\Delta$, is PL homemorphic to the dual complex of a (and therefore any) minimal dlt model of $(X,\Delta)$. Our strategy of proof is to decompose $\mathcal{D}(\Delta')$ using the geometry of $X$, as we did for $\mathcal{C}(\Delta)$. Each point in the dual complex $\mathcal{D}(\Delta')$ can be assigned uniquely up to scaling valuation on the function field $K(X)$. Every such valuation is centered on one of the lc centers of $X$, and two valuations belonging to the same open simplex always have the same center on $X$ because $f \colon X' \to X$ is a morphism.

\begin{definition}
Let $W$ be a subset of the lc centers of $(X,\Delta)$. Define $\mathcal{D}_{W}(X,\Delta)$ to be the union of open simplices corresponding to valuations centered along any lc center belonging to $W$.
\end{definition}

As $\mathcal{D}_{W}(X,\Delta)$ is determined by the description of $\mathcal{D}(\Delta)$ in terms of valuations, it is independent of the choice of dlt minimal model. The case where $W$ consists of all lc centers of dimension at least $1$ is especially important, we will denote this subcomplex by $\mathcal{D}_1(X,\Delta)$.

\begin{lemma}\label{graphs}
Let $(X,\Delta)$ be a threefold slc pair such that every coefficient of $\Delta$ is $1$ Let $W$ be the subset of lc centers of $(X,\Delta)$ consisting of a one dimensional lc center $\Gamma$ along with all components of $\Delta$ containing $\Gamma$.
The only possibilities for $T_\Gamma=\mathcal{D}_W(X,\Delta)$ are the following:

\begin{enumerate}
\item There are exactly two components $\Delta_1$ and $\Delta_2$ containing $\Gamma$, and $T_\Gamma$ is a line segment whose endpoints correspond to $\Delta_1$ and $\Delta_2$.
\item There is exactly one component $\Delta_1$ containing $\Gamma$, the preimage of $\Gamma$ in the normalization of $\Delta_1$ has two irreducible components and $T_\Gamma$ is a circle.

\item There is exactly one component $\Delta_1$ containing $\Gamma$, the preimage of $\Gamma$ in the normalization of $\Delta_1$ has one irreducible component and $T_\Gamma$ is a line segment with one endpoint corresponding to $\Delta_1$.
\end{enumerate}
\end{lemma}

\begin{proof}
We will use the fact that for some $0 \leq \Delta' \leq \Delta$, $(X,\Delta')$ is klt because a klt surface pair is always $\mathbb{Q}$-factorial\cite[Proposition 4.11]{KollarMori}. Hence in a neighborhood of $\Gamma$, $X$ is $\mathbb{Q}$-factorial. Thus if two components $\Delta_1$ and $\Delta_2$ contain $\Gamma$, then $(X,\Delta)$ is qdlt near $\Gamma$ and hence the dual complex is a segment connecting two points corresponding to the two divisors.

In any case we note that the complex $T_\Gamma$ is a graph, and the subset $\mathcal{D}_\Gamma(X,\Delta)$ is connected. Consider a dlt minimal model $(X',\Delta')$. Let $\Gamma$ be a component of $\Delta'$ which dominates $\Gamma$, and let $B$ be the rest of the boundary restricted to $\Gamma$. The degree of $\Gamma$ as a vertex of $T_\Gamma$ is the number of components of $B$ dominating $\Gamma$, but $-(K_{\Gamma}+B)$ is trivial over $\Gamma$, so by the Shokurov-Koll{\'a}r connectedness theorem, there are at most two such components.

Hence the graph $T_\Gamma$ is either a segment or a circle. If the preimage of $\Gamma$ in the normalization of $\Delta_1$ has two components, then since $\mathcal{D}_\Gamma(X,\Delta)$, the complement of $v(\Delta_1)$ in $C$, is connected, we must have that $T_\Gamma $ is a circle.

Otherwise, the preimage of $\Gamma$ has one component and so $v(\Delta_1)$ is a vertex of degree $1$, so $C$ is a line segment.
\end{proof}

Consider a threefold log canonical pair $(X,\Delta)$ such that $(D,B)$ is snc. 
\begin{proposition}\label{halfedgeisom}
There exists a PL homeomorphism $\theta \colon \mathcal{C}_1(X,\Delta) \to \mathcal{D}_1(X',\Delta')$ such that 
\begin{enumerate}
\item For any $0$-center $\Delta_i$, $\theta$ sends the normalization $D_i$ to the strict transform in $\Delta'$.
\item For any $1$-center $\Gamma$, the locus of $\mathcal{D}_1(X',\Delta)$ consisting of valuations supported on $\Gamma$ is $\theta$ of the union of the positive half edges corresponding to curves $C_i$ dominating $\Gamma$.
\item With $B_i \subset D_i$ and $\Gamma$ as above, if $\Delta'_i$ is the strict transform of $\Delta_i$ in $X'$ and $B'_i$ the strict transform of $B_i$ in $\Delta'_i$, then the preimage under $\theta$ of valuations supported on $B'_i$ has non empty intersection with any neighborhood of $0$ in the half edge corresponding to $B_i$.
\end{enumerate}
\end{proposition}

\begin{proof}
We use Lemma \ref{graphs} to endow $\mathcal{D}_1(X',\Delta')$ with the structure of a half edge graph as follows. The points are components of $D$, and the edges with $L(0)=D_i$ are curve strata with $C_j \subset D_i$. Two such half edges satisfy $L_{C_j}(1)=L_{C_{j'}}(1)$ if and only if $C_j$ and $C_{j'}$ are identified by the involution $\iota$. For all edges, we may choose the division into half edges such that for any dlt minimal model $(X', \Delta')$ the edge corresponding to the strict transform of $C_j$ intersects the image of $L_{C_j}$ in an interval containing $0$.

In each case we have that this half edge structure agrees with the half edge structure on $\mathcal{C}_1(X,\Delta)$, and the final condition guarantees a consistent choice of orientation on loops from a vertex to itself.

\end{proof}

To complete the proof of Theorem \ref{dmr} we must attach to $\mathcal{D}_1(X',\Delta')$ the simplices corresponding to centers dominating $0$-dimensional lc centers of $(X,\Delta)$. Let $P$ be such a center.

\begin{lemma}\label{zerocenters}
Let $W$ be an lc center of $(X',\Delta')$ whose image in $X$ is $Z$. Then $W$ contains a $0$-dimensional lc center of $(X',\Delta')$.
\end{lemma}

\begin{proof}
Since the image of any $0$-dimensional lc center is a $0$ dimensional lc center, we may assume that the preimage of $Z$ in $X'$ is a divisor, and that $W$ is also a divisor. Let $\Delta'_0$ be the strict transform of some component $\Delta_0$ of $\Delta$ which contained $P$. Then since the fibers of $X' \to X$ are connected, there is a finite chain of divisors $\Delta'_0=W_0, W_1 \ldots W_n = W$ such that each successive pair of divisors intersects in a curve.

It suffices to show for each pair $W_i, W_{i+1}$ that $W_i$ contains a $0$-dimensional lc center along the curve of intersection. For $\Delta'_0$, this curve is an exceptional curve over $D_0$, the normalization of $\Delta_0$. \cite[Proposition 4.6]{KollarBook}. Hence it contains some $0$-dimensional lc center. Thus $W_1$ satisfies the conclusion of the lemma.

Now suppose $W_i$ satisfies the lemma. We have that $K_{W_i}+C_i$ is numerically trivial where $C_i$ is the union of all $1$ dimensional lc centers contained in $W_i$. Hence by \cite{KollarXu} the dual complex is either $S^1$ or an interval. In either case every curve in $W_i$ contains a $0$-dimensional lc center. Hence $W_{i+1}$ contains a $0$-dimensional lc center and the lemma follows by induction.

\end{proof}

Since the PL homeomorphism type of $\mathcal{D}(\Delta')$ does not depend on the choice of minimal model, we may assume that the preimage of every lc center of $(X,\Delta)$ is a divisor. Let $\mathcal{B}$ be the flag complex of $\mathcal{D}(\Delta')$, which is the first barycentric subdivision. Let $Z$ be a log canonical center of dimension $0$. Denote by $\mathcal{B}_Z$ the subcomplex whose simplices correspond to chains of strata in $(X',\Delta')$ which all dominate $Z$.

\begin{proposition}
Under the assumptions of Theorem \ref{dmr}, the complex $\mathcal{B}_Z$ is PL homeomorphic to a closed $2$-disk.
\end{proposition}

\begin{proof}
As $\Delta$ is Cartier, $X$ is log terminal near $Z$. Hence we may apply \cite[Theorem 2]{deFernexKollarXu2012} to see that $\mathcal{B}_Z$ is contractible. Now, for any lc center of $(X',\Delta')$ whose image is $Z$, there is some divisor $D$ containg that center whose image is also $Z$. We will now show that $\mathcal{B}_Z$ is a manifold with boundary.

The vertices of $\mathcal{B}_Z$ correspond precisely to the lc centers of $(X', \Delta')$ with image $Z$. We show that the link of $\mathcal{B}_Z$ at each such center $W$ is either an interval or a circle.

If $W$ is a divisor $D_W$, the link is the dual graph of the corresponding log Calabi-Yau surface pair. The divisor $D_W$ contains a $0$-dimensional lc center so the dual complex has dimension $1$, so it is either an interval or $S^1$ \cite{KollarXu}.

If $W$ is a curve $C_W$ there again must be a $0$-center contained in $C_W$. The link is the following graph: The vertices are divisors containing $C_W$ exceptional over $z$ along with $0$-centers contained in $C_W$. The graph is complete bipartite; each divisor is connected to each $0$-center by an edge. As there are at most two divisors exceptional over $z$ containing $C_W$, and at most two $0$-centers contained in $C_W$, we must have that the link is either an interval or $S^1$.

Finally we consider the case where $W$ has dimension $0$. Then the link is the graph whose vertices are other lc centers dominating $Z$ which contain $W$. Two vertices are connected if they correspond to a curve contained in a divisor. This is a subgraph of the $6$ vertex cycle which contains at least one edge, so it suffices to show the graph is connected. This is only possible if there is a curve dominating $z$ not contained in a divisor dominating $Z$. But this is not true by assumption.

We conclude that $\mathcal{B}_Z$ is a contractible PL surface with boundary, hence a closed $2$-disk.

\end{proof}

Note that by Lemma \ref{zerocenters} every lc center dominating $Z$ contains a $0$-dimensional lc center. Hence $\mathcal{B}_Z$ is the union of $2$-simplices. Now, given $(X',\Delta')$ as above, we define $\mathcal{A}_Z$ to be the abstract union of $2$-simplices $Q_k \subset C_j \subset D_i$, where $Q_k \subset C_j \subset D_k$ and $Q_{k'} \subset C_{j'} \subset D_{k'}$ intersect along a subchain if they do in $\mathcal{D}(\Delta')$ and some stratum in the common subchain dominates $Z$.

\begin{proposition}
$\mathcal{A}_Z$ is a closed PL $2$-disk. 
\end{proposition}

\begin{proof}
If we perform blowups of $(X',\Delta')$ to perform a barycentric subdivision of the dual complex, the PL type of $\mathcal{A}_Z$ does not change, we again take a barycentric subdivision. In contrast, $\mathcal{B}_Z$ is replaced by a larger complex which is still homeomorphic to a disk. Morever every point of the interior of a $2$-simplex of $\mathcal{A}_Z$ is eventually contained in the interior of $\mathcal{B}_Z$.

Now let $\mathcal{A}'_Z$ be the locus of all points in $\mathcal{A}_Z$ where the complex is not locally homeomorphic to the open $2$-disk, that is, not a manifold. We will show that for these points the link is a closed half space, so that $\mathcal{A}$ is a manifold with boundary. If $x \in \mathcal{A}'_Z$ is contained in some $\mathcal{B}_Z$ then this must be so since $\mathcal{B}_Z$ is a manifold with boundary. Otherwise $x$ is contained in a simplex corresponding to a chain of strata which do not dominate $Z$.

If $x$ is a length $2$ chain then it corresponds to an edge of a $2$-simplex contained in no other $2$-simplex of $\mathcal{A}_Z$. Otherwise $x$ consists of a single stratum. The link at $x$ consists of the attachment graph of all the $2$-simplices containg the stratum $x$. But this is a graph which is a manifold with boundary because each $\mathcal{B}_Z$ is a manifold with boundary. Hence $\mathcal{A}_Z$ is a manifold with boundary.

Now, $\mathcal{A}'_Z$ is some graph which is contained in the boundary of $\mathcal{A}_Z$. Let $\mathcal{E}_Z$ be a closed infinitessimal neighborhood of $\mathcal{A}'_Z$, which must be isomorphic to $\mathcal{A}'_Z \times [0,1]$. For a sufficiently large $\mathcal{B}_Z$ this neighborhood intersects $\mathcal{B}_Z$ in a neighborhood of the boundary isomorphic to $\mathcal{A}'_Z \times [0,1]$. The resulting gluing is a closed PL $2$-disk, so $\mathcal{A}_Z$ is a closed PL $2$-disk.

\end{proof}





We are now ready to prove Theorem \ref{dmr}.

\begin{proof}[Proof of Theorem \ref{dmr}]
We start with the isomorphism $\theta \colon \mathcal{C}_1(\Delta) \to \mathcal{D}_1(X',\Delta')$ from Prop \ref{halfedgeisom}. By Prop \ref{attachment} the complex $\mathcal{C}(\Delta)$ is obtained from $\mathcal{C}_1(\Delta)$ by gluing, for each $0$-dimensional center $Z$, a $2$-disk along the path $\sigma_Z$. If $\sigma_Z$ is a cycle, we glue along the boundary of the disk, and if $\sigma_Z$ is a line, we glue along only part of the boundary.

To obtain $\mathcal{D}(\Delta')$ from $\mathcal{D}_1(\Delta')$ we must likewise include all of the strata dominating a $0$-center $Z$. The closure of all such strata inside $\mathcal{D}(\Delta')$ is the image of the PL $2$-disk $\mathcal{A}_Z$, which is attached by the image of the strata dominating other centers. Let $\mathcal{F}_Z$ be the subcomplex of $\mathcal{A}_Z$ corresponding whose image in $\mathcal{D}(\Delta')$ is contained in $\mathcal{D}_1(X,\Delta)$. The complex $\mathcal{F}_Z$ is a graph, contained in the boundary of the disk $\mathcal{Z}$ Our task is to identify the map from $\mathcal{F}_Z$ to $\mathcal{D}_1(X',\Delta')$ with $\sigma_Z$ via $\theta^{-1}$. 

First we will identify $\mathcal{F}_Z$ with $\Gamma(Z)$. The points of $\mathcal{F}_Z$ which dominate divisors in $(X,\Delta)$ each correspond to a chain $Q_k \subset D_i$ of strata of the normalization of $\Delta$.


The strict transforms in $(X', \Delta')$ of the two curves $C_{j}, C_{j'}$ containing $Q_k$ give two half edges in $\mathcal{A}'_P$ based at $Q_k \subset D_i$. Two such half edges $D_i \subset C_j$, $D_{i'} \subset C_{j'}$ meet at $0$ in $\mathcal{F}_Z$ if and only if $D_i = D_{i'}$ and $Q_k$ is the intersection of $C_{j}$ and $C_{j'}$. Two of these half edges can be extended in the positive direction to meet at $1$ if and only if $\iota$ sends $C_{j}$ to $C_{j'}$ and $\iota|_{C_j}(Q_k)=Q_{k'}$.
Hence we have produces a copy of $\Gamma(Z)$ inside $\mathcal{F}_Z$. Since the half edge $Q_k \subset C_j$ is sent by the attaching map to the half edge corresponding to $C_j \subset D_i$, we have for suitable choice of $\theta$ that the gluings agree.

It remains to show that all of $\mathcal{F}_Z$ is contained in this copy of $\Gamma(P)$. We have exhausted all vertices of $\mathcal{Z}_P$ corresponding to strict transforms of divisors of $\Delta$, the ones that remain dominate some curve. Given such a vertex there is a path of edges in $\mathcal{D}_1(X',\Delta')$ leading to the strict transform of a divisor in $\Delta$.

Following an edge consists of the following: If the original stratum is a curve, we choose one of the divisors containing it. If it is a divisor, we choose one of the curves in dominating the $1$-dimensional lc center. In either case, the two strata contain a $0$-center mapping to $Z$. Following the path gets us back to a chain $Q_k \subset D_j$, hence all strata and edges of this construction were contained in $\Gamma(Z)$.

The theorem follows.
\end{proof}

\begin{example}
Let $X$ be the cone over the second Veronese embedding of $\mathbb{P}^2$, given by 
\[
\Spec \mathbb{C}[x^2, xy, y^2, xz, z^2, yz, z^2].\]
 The affine semi-log canonical surface 
\[S= \Spec \mathbb{C}[s^2, st, t^2, s^2t, st^2]
\]
of Example \ref{doubleumbrella} embeds as a divisor $\Delta$ via the ring homomorphism
\[
x^2 \to s^2, xy \to st, y^2 \to t^2, xz \to s^2t, yz \to st^2, z^2 \to s^2t^2.
\]

It is straightforward to check that the pair $(X,\Delta)$ is log canonical. Theorem \ref{dmr} shows that $\mathcal{DMR}(X,\Delta)$ is PL homeomorphic to $\mathcal{C}(\Delta)$, which is a closed $2$-disk.

\end{example}
\section{Application to Degenerations of $I$-surfaces}\label{FPRexample}
The modular interpretation of the boundary strata of the moduli space $\mathcal{M}_g$ of curves of genus $g>2$ is well understood. The singular curves appearing in the boundary are exactly the connected stable curves with nodal singularities. Every component of geometric genus $0$ must have at least three nodal points, and every component of geometric genus $1$ must have at least one nodal point.

In higher dimensions little is known about the boundary strata of the KSBA moduli spaces. While every stable curve may be degenerated to a smooth curve, this is no longer true even for semi-log canonical surfaces with ample log canonical divisor. Thus it is an open problem to describe which singular surfaces appear as limits of smooth surfaces of a given topological type.

Consider a smooth surface $S$ with $p_g=2$, $q=0$ and $K^2=1$, called an I-surface. It is a classical result that such a surface can be realized as a degree $10$ hypersurface in the weighted projective space $\mathbb{P}(1,1,2,5)$. Equivalently $S$ is the double cover of a singular quadric cone $Q \subset \mathbb{P}^3$, branched along the intersection of $Q$ with a degree $5$ hypersurface. In a series of papers Franciosi, Pardini, and Rollenske \cite{FPR2015a,FPR2015b,FPR2017} give an account of the stable surfaces appearing in this moduli space.

We consider one of their examples \cite{FPR2015a}. As stable surfaces are semi-log canonical, they can be recovered from the data of a log canonical pair $(D,B)$ along with an involution on the normalization of $B$ preserving the log canonical centers. Take $D=\mathbb{P}^2$, and $B$ the be the union of $4$ lines in general position, numbered $L_1, \ldots L_4$. The involution consists of a pair of isomorphisms, without loss of generality between $L_1$ and $L_2$ and between $L_3$ and $L_4$. An isomorphism between smooth rational curves is determined uniquely by the choice of the image of three points. Hence the data of the involution consists of a pair of bijections between three element sets. For the surface $X_{3,1}$, we take

\[
\phi_{12}= \begin{bmatrix} P_{12} & P_{13} & P_{14} \\
P_{21} & P_{24} & P_{23} \\ \end{bmatrix}; \phi_{34} = \begin{bmatrix} P_{31} & P_{32} & P_{34}\\
P_{41} & P_{42} & P_{43} \\ \end{bmatrix}.
\]

The choice of $4$ lines on $\mathbb{P}^2$ guarantees that $K^2=1$ for the resulting stable surface. $X_{3,1}$ is the unique such surface up to isomorphism satisfying $q=0$ and $p_g=2$. In fact, $X_{3,1}$ embeds as a hypersurface of degree $10$ in $\mathbb{P}(1,1,2,5)$, and hence can be seen as a limit of smooth I-surfaces \cite[Theorem 3.3]{FPR2017}.

\begin{example}
Let $X \to B$ be a semistable $1$-parameter degeneration of I-surfaces with special fiber $X_0 \cong X_{3,1}$. The dual complex $\mathcal{C}(X_0)$ is the union of two $2$-spheres, glued along the union of two disks, and is homotopy equivalent to the wedge of two $2$-spheres.
\end{example}

Two constructions of this complex are given in Figure \ref{fig:FPR}.
By Theorem \ref{dmr}, the complex $\mathcal{C}(X_0)$ is PL homeomorphic to the dual complex of log canonical centers of any dlt minimal model of $(X,X_0)$.

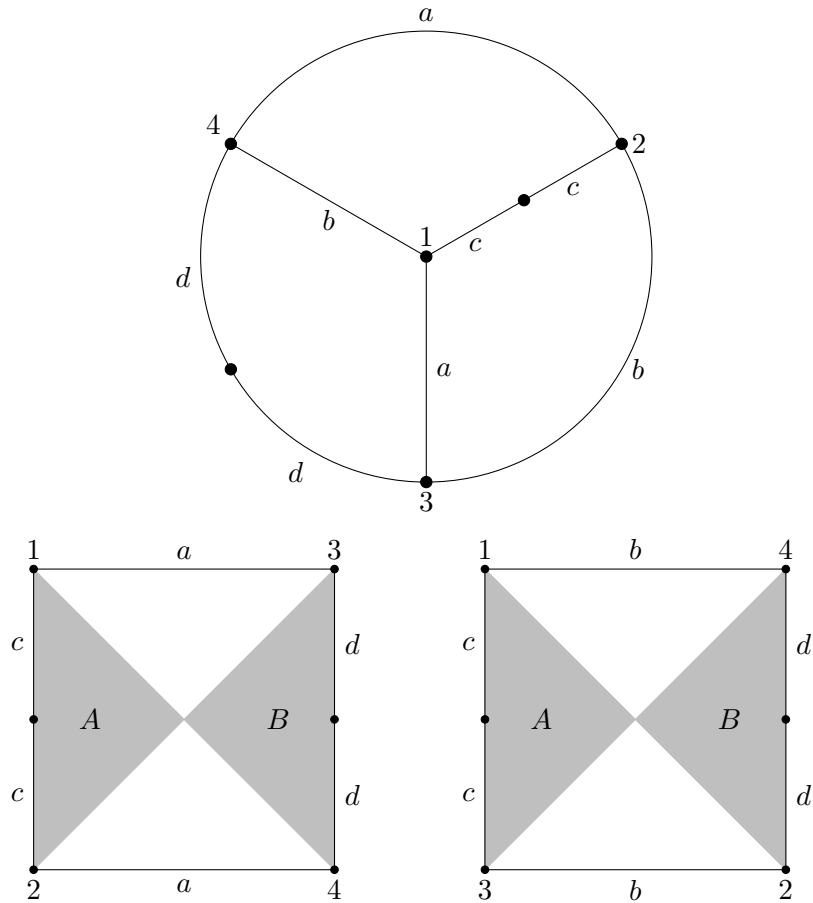
\begin{figure}
\begin{center}

\begin{tikzpicture}[scale=3]
\coordinate (1) at (0, 0);
\coordinate (2) at (0,-1);
\coordinate (3) at (0.86602, 0.5);
\coordinate (4) at (-0.86602, 0.5);

\coordinate (14) at (-0.86602/2, 0.25);
\coordinate (12) at (0, -.5);
\coordinate (13) at (0.86602/2, 0.25);
\coordinate (23) at (0.86602, -.5);
\coordinate (24) at (-0.86602, -0.5);
\coordinate (34) at (0, 1);

\coordinate (114) at (-0.86602/4, 0.25/2);
\coordinate (144) at (-3*0.86602/4, 3*0.25/2);
\coordinate (112) at (0, -.25);
\coordinate (122) at (0, -.75);
\coordinate (113) at (0.86602/4, 0.25/2);
\coordinate (133) at (3*0.86602/4, 0.75/2);
\coordinate (223) at (0.5,-0.86602);
\coordinate (233) at (1,0);
\coordinate (224) at (-.5,-0.86602);
\coordinate (244) at (-1,0);
\coordinate (334) at (0.5,0.86602);
\coordinate (344) at (-.5,0.86602);

\draw (1)--(2);
\draw (1)--(4);
\draw (1)--(3);
\draw (1) circle (1);

\draw (1) node[anchor=south] {$1$};
\draw (2) node[anchor=north] {$3$};
\draw (3) node[anchor=west] {$2$};
\draw (4) node[anchor=south east] {$4$};
\draw (12) node[anchor=west] {$a$};
\draw (113) node[anchor=north] {$c$};
\draw (133) node[anchor=north] {$c$};
\draw (14) node[anchor=north] {$b$};
\draw (23) node[anchor=west] {$b$};
\draw (34) node[anchor=south] {$a$};
\draw (244) node[anchor=north east] {$d$};
\draw (224) node[anchor=north east] {$d$};

\filldraw [black] 
(1) circle (0.025)
(2) circle (0.025)
(3) circle (0.025)
(4) circle (0.025)
(13) circle (0.025)
(24) circle (0.025);

\end{tikzpicture}
\end{center}

\begin{center}
\begin{tikzpicture}
\coordinate (50) at (-3, 2);
\coordinate (51) at (-5,4);
\coordinate (53) at (-5,0);
\coordinate (52) at (-1, 4);
\coordinate (54) at (-1, 0);
\coordinate (5013) at (-4.25, 2);
\coordinate (5024) at (-1.75, 2);
\coordinate (513) at (-5,2);
\coordinate (524) at (-1,2);
\coordinate (512) at (-3, 4);
\coordinate (534) at (-3, 0);
\coordinate (5224) at (-1,3);
\coordinate (5244) at (-1,1);
\coordinate (5113) at (-5,1);
\coordinate (5133) at (-5,3);

\filldraw[fill=gray!50!white, draw=gray!50!white] (50)--(51)--(53)--(50);
\filldraw[fill=gray!50!white, draw=gray!50!white] (50)--(52)--(54)--(50);

\draw (5013) node {$A$};
\draw (5024) node {$B$};

\filldraw [black]
(51) circle (0.05)
(52) circle (0.05)
(53) circle (0.05)
(54) circle (0.05)
(524) circle (0.05)
(513) circle (0.05);

\draw (51)--(52);
\draw (51)--(53);
\draw (52)--(54);
\draw (53)--(54);

\draw (51) node[anchor=south] {$1$};
\draw (52) node[anchor=south] {$3$};
\draw (53) node[anchor=north] {$2$};
\draw (54) node[anchor=north] {$4$};

\draw (512) node [anchor=south] {$a$};
\draw (534) node [anchor=north] {$a$};
\draw (5224) node [anchor=west] {$d$};
\draw (5244) node [anchor=west] {$d$};
\draw (5113) node [anchor=east] {$c$};
\draw (5133) node [anchor=east] {$c$};

\coordinate (60) at (3,2);
\coordinate (64) at (5,4);
\coordinate (62) at (5,0);
\coordinate (61) at (1, 4);
\coordinate (63) at (1, 0);
\coordinate (6024) at (4.25, 2);
\coordinate (6013) at (1.75, 2);
\coordinate (624) at (5,2);
\coordinate (613) at (1,2);
\coordinate (612) at (3, 4);
\coordinate (634) at (3, 0);
\coordinate (6224) at (5,3);
\coordinate (6244) at (5,1);
\coordinate (6113) at (1,1);
\coordinate (6133) at (1,3);

\filldraw[fill=gray!50!white, draw=gray!50!white] (60)--(61)--(63)--(60);
\filldraw[fill=gray!50!white, draw=gray!50!white] (60)--(62)--(64)--(60);

\draw (6013) node {$A$};
\draw (6024) node {$B$};

\filldraw [black]
(61) circle (0.05)
(62) circle (0.05)
(63) circle (0.05)
(64) circle (0.05)
(624) circle (0.05)
(613) circle (0.05);

\draw (61)--(63);
\draw (61)--(64);
\draw (62)--(64);
\draw (63)--(62);

\draw (612) node [anchor=south] {$b$};
\draw (634) node [anchor=north] {$b$};
\draw (6224) node [anchor=west] {$d$};
\draw (6244) node [anchor=west] {$d$};
\draw (6113) node [anchor=east] {$c$};
\draw (6133) node [anchor=east] {$c$};

\draw (61) node[anchor=south] {$1$};
\draw (62) node[anchor=north] {$2$};
\draw (63) node[anchor=north] {$3$};
\draw (64) node[anchor=south] {$4$};
\end{tikzpicture}
\end{center}
\caption{Above: The complex $\mathcal{C}(X_{3,1})$ is obtained by taking the cone over the graph $K_4$ and identifying points and edges of the resulting complex according to the labeling. Below: The same complex can be seen as the union of two spheres $S^2$ glued along the labeled shaded areas.}
\label{fig:FPR}
\end{figure}

\bibliographystyle{amsalpha}      
\bibliography{References}
\end{document}